%% file: toe.tex
\documentclass{amsart}

\usepackage{times}
\usepackage{hyperref}

\newcommand{\ToE}[1]{(\hyperref[ToE#1]{ToE#1})}
\newcommand{\ToEp}{(\hyperref[ToE2']{ToE2'})}

\input skip-thm

\input qformsl-def

\renewcommand{\D}{\mathrm{D}}

\newcommand{\B}{\mathrm{B}}
\newcommand{\E}{\mathrm{E}}
\renewcommand{\F}{\mathrm{F}_4}
\renewcommand{\G}{\mathrm{G}_2}
\newcommand{\U}{\mathrm{U}}

\renewcommand{\H}{\mathbb{H}}

\newcommand{\cE}{\E}

\newcommand{\SM}{\mathrm{SM}}

\newcommand{\GUT}{\mathrm{GUT}}

\newcommand{\g}{\mathfrak{g}}

\renewcommand{\sl}{\mathfrak{sl}}
\renewcommand{\sp}{\mathfrak{sp}}

\newcommand{\h}{\mathfrak{h}}
\renewcommand{\e}{\mathfrak{e}}

\newcommand{\at}{\tilde{\alpha}}

\newcommand{\ach}{\alpha^\vee}
\newcommand{\bch}{\beta^\vee}
\newcommand{\Qch}{Q^\vee}

\newcommand{\mx}{{\mathrm{max}}}

\numberwithin{equation}{section}

\begin{document}

\title{There is no ``Theory of Everything" inside $\E_8$}
\author{Jacques Distler}
\address{(Distler) Theory Group, Department of Physics, and Texas Cosmology Center, University of Texas, Austin, TX 78712}
\email{distler@golem.ph.utexas.edu}
\thanks{Report \#: UTTG-05-09, TCC-020-09}

\author{Skip Garibaldi}
\address{(Garibaldi) Department of Mathematics \& Computer Science, 400 Dowman Dr., Emory University, Atlanta, GA 30322}
\email{skip@member.ams.org}


%\thanks{{\tt Version of \today.}}

\begin{abstract}
We analyze certain subgroups of real and complex forms of the Lie group $\E_8$, and deduce that any ``Theory of Everything" obtained by embedding the gauge groups of gravity and the Standard Model into a real or complex form of $\E_8$ lacks certain representation-theoretic properties required by physical reality.  The arguments themselves amount to representation theory of Lie algebras in the spirit of Dynkin's classic papers and are written for mathematicians.
\end{abstract}

\maketitle

%%%%%%%%%%%%%%%%%%%%%%%%%%%%%%%%%%%%%%%%%%%%%%%%%%%%%%%%%%%%%%%%%
\section{Introduction}

Recently, the preprint \cite{Lisi} by Garrett Lisi has generated a lot of popular interest. It boldly claims to be a sketch of a ``Theory of Everything", based on the idea of combining the local Lorentz group and the gauge group of the Standard Model in a real form of $\E_8$ (necessarily not the compact form, because it contains a group isogenous to $\SL(2, \C)$).  The purpose of this paper is to explain some reasons why an entire class of such models---which include the model in \cite{Lisi}---cannot work, using mostly mathematics with relatively little input from physics.

%Below, the even-numbered sections are purely mathematical and the odd-numbered sections have more physics content.  We first consider real forms of $\E_8$ in sections \ref{real.math.sec} and \ref{real.phys.sec} and then consider the complex $\E_8$ (as a real Lie group of dimension $2 \times 248$) in sections \ref{c.math.sec} and \ref{c.phys.sec}.

The mathematical set up is as follows.  Fix a real Lie group $\cE$.  We are interested in subgroups $\SL(2, \C)$ and $G$ of $\cE$ so that:
\[
\parbox{4in}{$G$ is connected, compact, and centralizes $\SL(2,\C)$} \tag{ToE1}\label{ToE1}
\]
We complexify and then decompose $\Lie(\cE) \ot \C$ as a direct sum of representations of $\SL(2,\C)$ and $G$.
We identify $\SL(2, \C) \otimes_\R \C$ with $\SL_{2,\C} \times \SL_{2,\C}$ and write
\begin{equation} \label{decomp}
\Lie(\cE) = \bigoplus_{m,n \ge 1} m \ot n \ot V_{m,n}
\end{equation}
where $m$ and $n$ denote the irreducible representation of $\SL_{2,\C}$ of that dimension and $V_{m,n}$ is a complex representation of $G \ot_\R \C$.  (Physicists would usually write $2$ and $\bar{2}$ instead of $2 \ot 1$ and $1 \ot 2$.)  Of course,
\[
\overline{m \ot n \ot V_{m,n}} \simeq n \ot m \ot \overline{V_{m,n}}
\]
and since the action of $\SL(2,\C) \cdot G$ on $\Lie(\cE)$ is defined over $\R$, we deduce that $\overline{V_{m,n}} \simeq V_{n,m}$.
We further demand that
\begin{gather*}
\parbox{4in}{$V_{m,n} = 0$ if $m + n > 4$, and} \tag{ToE2}\label{ToE2} \\
\parbox{4in}{$V_{2,1}$ is a complex representation of $G$.} \tag{ToE3} \label{ToE3}
\end{gather*}
We recall the definition of complex representation and explain the physical motivation for these hypotheses in the next section.  Roughly speaking, \ToE1 is a trivial requirement based on trying to construct a Theory of Everything along the lines suggested by Lisi, \ToE2 is the requirement that the model not contain any ``exotic" higher-spin particles, and \ToE3 is the statement that the gauge theory (with gauge group $G$) is \emph{chiral}, as required by the Standard Model.
In fact, physics requires slightly stronger hypotheses on $V_{m,n}$, for $m+n=4$. We will not impose the stronger version of \ToE2.

\begin{defn} \label{toe.def}
A \emph{candidate ToE subgroup} of a real Lie group $\cE$ is a subgroup generated by a copy of $\SL(2, \C)$ and a subgroup $G$ such that \ToE1 and (\hyperref[ToE3]{ToE2}) hold.  A \emph{ToE subgroup} is a candidate ToE subgroup for which \ToE3 also holds.
\end{defn}

Our main result is:

\begin{thm} \label{MT}
There are no ToE subgroups in (the transfer of) the complex $\E_8$ nor in any real form of $\E_8$.
\end{thm}

\subsection*{Notation}
Unadorned Lie algebras and Lie groups mean ones over the real numbers.  We use a subscript $\C$ to denote complex Lie groups---e.g., $\SL_{2,\C}$ is the (complex) group of 2-by-2 complex matrices with determinant 1.  We can view a $d$-dimensional complex Lie group $G_\C$ as a $2d$-dimensional real Lie group, which we denote by $R(G_\C)$.  (Algebraists call this operation the ``transfer" or ``Weil restriction of scalars"; geometers, and many physicists, call this operation ``realification.")  We use the popular notation of $\SL(2, \C)$ for the transfer $R(\SL_{2,\C})$ of $\SL_{2, \C}$; it is a double covering of the ``restricted Lorentz group", i.e., of the identity component $\SO(3,1)_0$ of $\SO(3,1)$.

\begin{borel}{Strategy and main results}\label{mainresults}
Our strategy for proving Theorem \ref{MT} will be as follows. We will first catalogue, up to conjugation, all possible embeddings of $\SL(2,\C)$ satisfying the hypotheses of \ToE2. The list is remarkably short.  Specifically, for every candidate ToE subgroup of $\cE$, the group $G$ is contained in the maximal compact, connected subgroup $G_\mx$ of the centralizer of $\SL(2, \C)$ in $\cE$.  The proof of Theorem \ref{MT} shows that the only possibilities are:
\begin{equation}\label{Results}
\begin{array}{ccc} 
\cE & G_\mx & V_{2,1}\\ \hline
\E_{8(-24)} & \Spin(11) & 32 \\
\E_{8(8)}&\Spin(5)\times \Spin(7) & (4,8)\\
\E_{8(-24)}&\Spin(9)\times \Spin(3) & (16,2) \\
R(\E_{8,\C})&\E_7 & 56\\
R(\E_{8,\C})& \Spin(12) & 32 \oplus 32'\\
R(\E_{8,\C})& \Spin(13) & 64
\end{array}
\end{equation}
We then note that the representation $V_{2,1}$ of $G_{\text{max}}$ (and hence, of any $G\subseteq G_{\text{max}}$) has a self-conjugate structure. In other words, \ToE3 fails.
\end{borel}

%%%%%%%%%%%%%%%%%%%%%%%%%%%%%%%%%%%%%%%%%%%%%%%%%%%%%%%%%%%%%%%%%
\section{Physics background} \label{phys.sec}

One of the central features of modern particle physics is that the world is described by a \emph{chiral gauge theory}.

\begin{borel*} \label{fermions}
Let $M$ be a four-dimensional pseudo-Riemannian manifold, of signature $(3,1)$, which we will take to be oriented, time-oriented and spin. Let $G$ be a compact Lie group. The data of a \emph{gauge theory on $M$ with gauge group $G$} consists of a connection, $A$, on a principal $G$-bundle, $P\to M$, and some ``matter fields" transforming as sections of vector bundle(s) associated to unitary representations of $G$.

Of particular interest are the \emph{fermions} of the theory. The orthonormal frame bundle of $M$ is a principal $\SO(3,1)_0$ bundle. A choice of spin structure defines a lift to a principal $\Spin(3,1)_0 = \SL(2,\mathbb{C})$ bundle.  Let $S_\pm \to M$ be the irreducible spinor bundles, associated, via the defining two-dimensional representation and its complex conjugate, to this $\SL(2,\mathbb{C})$ principal bundle.

The \emph{fermions of our gauge theory} are denoted
\[
  \psi \in \Gamma(S_+\otimes V),\qquad \overline{\psi} \in \Gamma(S_-\otimes \overline{V})
\]
where $V\to M$ is a vector bundle associated to a (typically reducible) representation $R$ of $G$.
\end{borel*}

\begin{defn}\label{physreps}
Consider, $V$, a unitary representation of $G$ over $\C$---i.e., a homomorphism $G \ra \U(V)$---and an antilinear map $J \!: V \ra V$ that commutes with the action of $G$.  The map $J$ is called a \emph{real structure} on $V$ if $J^2 =1$; physicists call a representation possessing a real structure \emph{real}.  The map $J$ is called a \emph{quaternionic structure} on $V$ if $J^2 = -1$; physicists call a representation possessing a quaternionic structure \emph{pseudoreal}.

Subsuming these two subcases, we will say that $V$ \emph{has a self-conjugate structure} if  there exists an antilinear map $J: V\to V$ commuting with the action of $G$ and satisfying $J^4 =1$.
Physicists call a representation $V$ that does not possess a self-conjugate structure \emph{complex}.
\end{defn}

\begin{rmk}
We sketch how to translate the above definition into the language of algebraic groups and Galois descent as in \cite{Borel} and \cite[\S{X.2}]{SeLF}.  Let $G$ be an algebraic group over $\R$ and fix a representation $\rho \!: G \ot \C \ra \GL(V)$ for some complex vector space $V$.  Let $J$ be an antilinear map $V \ra V$ that satisfies
\begin{equation} \label{JG}
\rho(g) = J^{-1} \rho(\overline{g}) J \quad \text{for $g \in G(\C)$.}
\end{equation}
We define real, quaternionic, etc., by copying the second and third sentences verbatim from Definition \ref{physreps}.

(In the special case where $G$ is compact, there is necessarily a positive-definite invariant hermitian form on $V$ and $\rho$ arises by complexifying some map $G \ra \U(V)$; this puts us back in the situation of Def.~\ref{physreps}.  In the special case where $G$ is connected, the hypothesis from Def.~\ref{physreps} that $J$ commutes with $G(\R)$---which is obviously implied by \eqref{JG}---is actually equivalent to \eqref{JG}.  Indeed, both sides of \eqref{JG} are morphisms of varieties over $\C$, so if they agree on $G(\R)$---which is Zariski-dense by \cite[18.2(ii)]{Borel}---then they are equal on $G(\C)$.)

 If $V$ has a real structure $J$, then the $\R$-subspace $V'$ of elements of $V$ fixed by $J$ is a real vector space and $V$ is canonically identified with $V' \ot \C$ so that $J(v' \ot z) = v' \ot \overline{z}$ for $v' \in V'$ and $z \in \C$; this is Galois descent.  Because $\rho$ commutes with complex conjugation (which acts in the obvious manner on $G(\C)$ and via $J$ on $V$), it is the complexification of a homomorphism $\rho' \!: G \ra GL(V')$ defined over $\R$ by \cite[AG.14.3]{Borel}.  Conversely, if there is a representation $(V', \rho')$ whose complexification is $(V, \rho)$, then taking $J$ to be complex conjugation on $V = V' \ot \C$ defines a real structure on $(V, \rho)$.

If $V$ has a quaternionic structure $J$, then we define a real structure $\hat{J}$ on $\hat{V} := V \oplus V$ via $\hat{J} (v_1, v_2) := (J v_2, - J v_1)$.

Finally, suppose that $G$ is reductive and $V$ is irreducible (as a representation over $\C$, of course).  Then by \cite[\S7]{Ti:R}, there is a unique irreducible real representation $W$ whose complexification $W \ot \C$ contains $V$ as a summand.  By Schur, $\End_G(W)$ is a division algebra, and we have three possibilities:
\begin{itemize}
\item $\End_G(W) = \R$, $W \ot \C \simeq V$, and $V$ has a real structure.
\item $\End_G(W) = \H$, $W \ot \C \simeq V \oplus V$, and $V$ has a quaternionic structure.
\item $\End_G(W) = \C$, $W \ot \C \simeq V \oplus \overline{V}$ where $V \not\simeq \overline{V}$, and $V$ is complex.
\end{itemize}
We have stated this remark for $G$ a group over $\R$, but all of it generalizes easily to the case where $G$ is reductive over a field $F$ and is split by a quadratic extensions $K$ of $F$.
\end{rmk}

\begin{defn}\label{chiral}
A gauge theory, with gauge group $G$, is said to be \emph{chiral} if the representation $R$ by which the fermions \eqref{fermions} are defined is complex in the above sense. By contrast, a gauge theory is said to be \emph{nonchiral} if the representation $R$ in \ref{fermions} has a self-conjugate structure.

Note that whether a gauge theory is chiral depends crucially on the choice of $G$. A gauge theory might be chiral for gauge group $G$, but \emph{nonchiral} for a subgroup $H\subset G$.  That is, there can be a self-conjugate structure on $R$ compatible with $H$, even though no such structure exists that is compatible with the full group $G$.

Conversely, suppose that a gauge theory is nonchiral for the gauge group $G$. It is also necessarily nonchiral for any gauge group $H\subset G$.
\end{defn}

\subsection*{GUTs}
The Standard Model is a chiral gauge theory with gauge group 
\[
G_\SM := (\SU(3) \times \SU(2) \times \U(1)) / (\Zm6)
\]
Various grand unified theories (GUTs) proceed by embedding $G_\SM$ is some (usually simple) group, $G_\GUT$. Popular choices for $G_\GUT$ are  $\SU(5)$ \cite{GeorgiGlashow}, $\Spin(10)$, $\E_6$, and the Pati-Salam group, $(\Spin(6)\times \Spin(4))/ (\Zm2)$ \cite{PatiSalam}.

It is easiest to explain what the fermion representation of $G_\SM$ is after embedding $G_\SM$ in $G_\GUT := \SU(5)$. Let $W$ be the five-dimensional defining representation of $\SU(5)$. The representation $R$ from \ref{fermions} is the direct sum of three copies of
\[
   R_0=\wedge^2 W \oplus \overline{W}
\]
Each such copy is called a ``generation" and is 15-dimensional.  One identifies each of the 15 weights of $R_0$ with left-handed fermions: 6 quarks (two in a doublet, each in three colors), two leptons (e.g., the electron and its neutrino), 6 antiquarks, and a positron.  With three generations, $R$ is 45-dimensional.

\begin{defn}\label{ngen.def}
As a generalization, physicists sometimes consider the \emph{$n$-generation Standard Model}, which is defined in similar fashion, but with $R=R_0^{\oplus n}$. The $n$-generation Standard Model is a chiral gauge theory, for any positive $n$. Particle physics, in the real world, is described by ``the" Standard Model, which is the case $n=3$.
\end{defn}

For the other choices of GUT group, the analogue of a generation ($R_0$) is higher-dimensional, containing additional fermions that are not seen at low energies. When decomposed under $G_\SM\subset G_\GUT$, the representation decomposes as $R_0 + R'$, where $R'$ is a real representation of $G_\SM$. In $\Spin(10)$, a generation is the 16-dimensional half-spinor representation. In $\E_6$, it is a 27-dimensional representation, and for the Pati-Salam group it is the $(4,1,2)\oplus (\overline{4},2,1)$ representation. In each case, these representations are complex representations (in the above sense) of $G_\GUT$, and the complex-conjugate representation is called an ``anti-generation."

%%%%%%%%%%%%%%%%%%%%%%%%%%%%%%%%%%%%%%%%%%%%%%%%%%%%%%%%%%%%%%
\section{Lisi's proposal from \protect{\cite{Lisi}}}\label{LisiElaboration}
In the previous section, we have described a chiral gauge theory in a fixed (pseudo) Riemannian structure on $M$. Lisi's proposal \cite{Lisi} is to try to combine the spin connection on $M$ and the gauge connection on $P$ into a single dynamical framework.  This motivates Definition \ref{toe.def} of a ToE subgroup.

More precisely, following \cite{Lisi}, we fix
subgroups $\SL(2, \C)$ and $G$ --- say, with $G = G_\SM$ --- satisfying \ToE1 in some real Lie group $\cE$. The action of the central element $-1 \in \SL(2, \C)$ provides a $\mathbb{Z}/2\mathbb{Z}$-grading on the Lie algebra of $\cE$. This $\mathbb{Z}/2\mathbb{Z}$-grading allows one to define a sort of superconnection associated to $\cE$ (precisely what sort of superconnection is explained in a blog post by the first author \cite{Distler:SuperConnections}).  In the proposal of \cite{Lisi}, we are supposed to identify each of the generators of $\Lie(\cE)$ as either a boson or a fermion.  (See Table 9 in \cite{Lisi} for an identification of the 240 roots.)

The Spin-Statistics Theorem \cite{StreaterWightman:PCT} says that fermions transform as spinorial representations of $\Spin(3,1)$; bosons transform as ``tensorial" representations (representation which lift to the double cover, $\SO(3,1)$). To be consistent with the Spin-Statistics Theorem, we must, therefore, require that the fermions belong to the $-1$-eigenspace of the aforementioned $\mathbb{Z}/2\mathbb{Z}$ action, and the bosons to the $+1$-eigenspace.

In fact, to agree with \ref{fermions}, we should require that the $-1$-eigenspace (when tensored with $\mathbb{C}$) decompose as a direct sum of two-dimensional representations (over $\C$) of $\SL(2,\C)$, corresponding to ``left-handed" and ``right-handed" fermions, in the sense of \ref{fermions}.  

\subsection*{Interpretations of $V_{m,n}$ and \ToE2}\label{toe2exp}
In the notation of \eqref{decomp}, the $V_{m,n}$, with $m+n$ odd, correspond to fermions; those with $m+n$ even correspond bosons. In Lisi's setup, the bosons are 1-forms on $M$, with values in a vector bundle associated to the aforementioned $\Spin(3,1)_0$ principal bundle via the $m\otimes n$ representation (with $m+n$ even). The $V_{m,n}$ with $m+n=4$ are special; they correspond to the gravitational degrees of freedom in Lisi's theory.  $(3\otimes 1) \oplus (1\otimes 3)$ is the adjoint representation of $\SL(2,\C)$; these correspond to the spin connection. The 1-form with values in the $2\otimes 2$ representation is the vierbein\footnote{In making this identification, we have tacitly assumed that $V_{2,2}$ is one-dimensional. This is, in fact, required for a unitary interacting theory. We will not, however, impose this additional constraint. Suffice to say that it is not satisfied by any of the \emph{candidate ToE subgroups} (per Definition \ref{toe.def}) of $\E_8$.}. 

It is a substantial result from physics (see sections 13.1, 25.4 of \cite{Weinberg:QFT}) that a unitary interacting theory is incompatible with massless particles in higher representation ($m + n\ge 6$).  Our hypothesis \ToE2 reflects this and also forbids gravitinos ($m+n=5$). In \S\ref{supergravity}, we will revisit the possibility of admitting gravitinos.

\subsection*{Explanation of \ToE3}
Our hypothesis \ToE3 says that the candidate ``Theory of Everything" one obtains from subgroups $\SL(2,\C)$ and $G$ as in \ToE1 must be \emph{chiral} in the sense of Definition \ref{chiral}.
\footnote{Of course, there are many other features of the Standard Model that a candidate Theory of Everything must reproduce. We have chosen to focus on the  requirement that the theory be {chiral} for two reasons. First,  it is ``physically robust": Whatever intricacies a quantum field theory may possess at high energies, if it is \emph{non}-chiral, there is no known mechanism by which it could reduce to a \emph{chiral} theory at low energies (and there are strong arguments \cite{tHooft} that no such mechanism exists). Second,  chirality is easily translated into a mathematical criterion---our \ToE3. This allows us to study a purely representation-theoretic question and side-step the difficulties of making sense of Lisi's proposal as a dynamical quantum field theory.}

In private communication, Lisi has indicated that he objects to our condition \ToE3, because he no longer wishes to identify all 248 generators of $\Lie(\cE)$ as particles (either bosons or fermions). In his new---and unpublished---formulation, only a subset are to be identified as particles. In  particular, $V_{2,1}$ is typically a reducible representation of $G$ and, in his new formulation, only a subrepresentation corresponds to particles (fermions). This is not the approach followed in \cite{Lisi}, where all 248 generators are identified as particles and where, moreover, 20-odd of these are claimed to be new as-yet undiscovered particles---a prediction of his theory. As recently as April 2009, Lisi reiterated this prediction in an essay published in the \emph{Financial Times}, \cite{FinTimes}.

Our paper assumes that the approach of \cite{Lisi} is to be followed, and  that all 248 generators are to be identified as particles, hence \ToE3. In any case, even if one identifies only a subset of the generators as particles, all the fermions must come from the $(-1)$-eigenspace, which is too small to accommodate 3 generations, as we now show.

\subsection*{No-go based on dimensions}  \label{dim}
The fermions of Lisi's theory correspond to weight vectors in $V_{m,n}$, with $m+n$ odd. In particular, the weight vectors in $V_{2,1}$ and $V_{1,2}$ correspond  (as in \S\ref{fermions}) to left- and right-handed fermions, respectively.  Since there are $3\times 15= 45$ known fermions of each chirality, $V_{2,1}$ must be at least 45-dimensional, and similarly for $V_{1,2}$.
Thus, the $-1$-eigenspace of the central element of $\SL(2,\C)$, which contains
  $(2\otimes 1\otimes V_{2,1}) \oplus (1\otimes 2\otimes V_{1,2})$,
must have dimension at least $2 \times 2 \times 45 = 180$.

When $\cE$ is a real form of $\E_8$, the $-1$-eigenspace has dimension 112 or 128 (this is implicit in Elie Cartan's classification of real forms of $\E_8$ as in \cite[p.~518, Table V]{Helgason}),\footnote{Alternatively, Serre's marvelous bound on the trace from \cite[Th.~3]{Se:trace} or \cite[Th.~1]{EKV} implies that for every element $x$ of order 2 in a reductive complex Lie group $G$, the $-1$-eigenspace of $\mathrm{Ad}(x)$ has dimension $\le (\dim G + \rank G) / 2$.  In particular, when $G$ is a real form of $\E_8$, the $-1$-eigenspace has dimension $\le 128$.} so no identification of the fermions as distinct weight vectors in $\Lie(\cE)$ (as in Table 9 in \cite{Lisi}) can be compatible with the Spin-Statistics Theorem and the existence of three generations.

These dimensional considerations do not, however, rule out the possibility of accommodating a 1- or 2-generation Standard Model (per Definition \ref{ngen.def}) in a real form of $\E_8$. That requires more powerful considerations, which are the subject of our main theorem.  We now turn to the proof of that theorem.

%%%%%%%%%%%%%%%%%%%%%%%%%%%%%%%%%%%%%%%%%%%%%%%%%%%%%%%%%%%%%%%%%
\section{$\sl_2$ subalgebras and the Dynkin index}

\begin{borel*}
In \cite[\S2]{Dynk:ssub}, Dynkin defined the \emph{index} of an inclusion $f \!: \g_1 \injects \g_2$ of simple complex Lie algebras as follows.  Fix a Chevalley basis of the two algebras, so that the Cartan subalgebra $\h_1$ of $\g_1$ is contained in the Cartan subalgebra $\h_2$ of $\g_2$.  The Chevalley basis identifies $\h_i$ with the complexification $\Qch_i \ot \C$ of the coroot lattice $\Qch_i$ of $\g_i$, and the inclusion $f$ gives an inclusion $\Qch_1 \ot \C \injects \Qch_2 \ot \C$.  Fix the Weyl-invariant inner product $(\,,\,)_i$ on $\Qch_i$ so that $(\ach, \ach)_i = 2$ for short coroots $\ach$.  Then the \emph{Dynkin index} of the inclusion is the ratio $(f(\ach), f(\ach))_2 / (\ach, \ach)_1$ where $\ach$ is a short coroot of $\g_1$.  For example, the irreducible representation $\sl_2 \ra \sl_n$ has index $\binom{n+1}{3}$ by \cite[Eq.~(2.32)]{Dynk:ssub}. 
\end{borel*}

\begin{borel*} \label{sl2.h}
We now consider the case $\g_1 = \sl_2$ and write simply $\g$ and $\Qch$ for $\g_2$ and $Q_2^\vee$.  The coroot lattice of $\sl_2$ is $\Z$ and the image of 1 under the map $\Z \injects \Qch$ is an element $h \in \h$ called the \emph{defining vector} of the inclusion.
In \S8 of his paper (or see \cite[\S{VIII.11}]{Bou:g7}), Dynkin proved that, after conjugating by an element of the automorphism group of $\g$, one can assume that the defining vector $h$ satisfies the strong restrictions:
\[
h = \sum_{\delta \in \Delta} p_\delta \delta^\vee \quad \text{for $p_\delta$ real and non-negative \cite[Lemma 8.3]{Dynk:ssub},}
\]
where $\Delta$ denotes the set of simple roots of $\g$ and further that
\begin{equation} \label{dynk.norm}
\delta(h) \in \{ 0, 1, 2 \} \quad \text{for all $\delta \in \Delta$.}
\end{equation}
But note that for each simple root $\delta$, the fundamental irreducible representation of $\g$ with highest weight dual to $\delta^\vee$ restricts to a representation of $\sl_2$ with $p_\delta$ as a weight, hence \emph{$p_\delta$ is an integer.} 

As a consequence of these generalities and specifically \cite[Lemma 8.2]{Dynk:ssub}, one can identify an $\sl_2$ subalgebra of $\g$ up to conjugacy by writing the Dynkin diagram of $\g$ and putting the number $\delta(h)$ from \eqref{dynk.norm} at each vertex; this is the \emph{marked Dynkin diagram} of the $\sl_2$ subalgebra.

Here is an alternative formula for computing the index of an $\sl_2$ subalgebra from  its marked Dynkin diagram.  Write $\kappa_\g$ and $m^\vee$ for the Killing form and dual Coxeter number of $\g$.  We have:
\begin{equation} \label{index.alt}
(\text{Dynkin index}) = \frac12 (h, h) = \frac{1}{4m^\vee} \kappa_\g(h, h) = \frac{1}{2m^\vee} \sum_{\text{positive roots $\alpha$ of $\g$}} \alpha(h)^2,
\end{equation}
where the second equality is by, e.g., \cite[\S5]{GrossNebe}, and the third is by the definition of $\kappa_\g$.  One can calculate the number $\alpha(h)$ by writing $\alpha$ as a sum of positive roots and applying the marked Dynkin diagram for $h$.
\end{borel*}

\begin{lem} \label{conj1}
For every simple complex Lie algebra $\g$, there is a unique copy of $\sl_2$ in $\g$ of index $1$, up to conjugacy.
\end{lem}

This is (equivalent to) Theorem 2.4 in \cite{Dynk:ssub}.  We give a different proof for the convenience of the reader.

\begin{proof}
The index of an $\sl_2$-subalgebra is $(h,h)/2$, where the defining vector $h$ belongs to the coroot lattice $\Qch$.  If $\g$ is not of type B, then the coroot lattice is not of type C, and the claim amounts to the statement that the 
vectors of minimal length in the coroot lattice are actually coroots.  This follows from the constructions of the root lattices in \cite[\S12.1]{Hum:LA}. 

Otherwise $\g$ has type B and is $\so_n$ for some odd $n \ge 5$.  The conjugacy class of an $\sl_2$-subalgebra is determined by the restriction of the natural $n$-dimensional representation; they are parameterized by partitions of $n$ (i.e., $\sum n_i = n$) so that the even $n_i$ occur with even multiplicity and some $n_i > 1$, see \cite[5.1.2]{CMcG} or \cite[\S6.2.2]{OV3}.  The index of the composition $\sl_2 \ra \so_n \ra \sl_n$ is then $\sum \binom{n_i+1}{3}$; we must classify those partitions such that this sum equals the Dynkin index of $\so_n \ra \sl_n$, which is 2.  The unique such partition is $2 + 2 + 1 + \cdots + 1 > 0$.
\end{proof}

In the bijection between conjugacy classes of $\sl_2$ subalgebras and orbits of nilpotent elements in $\g$ from \cite[3.2.10]{CMcG}, the unique orbit of index 1 $\sl_2$'s corresponds to the minimal nilpotent orbit described in \cite[4.3.3]{CMcG}.

If $\g$ has type C, $\F$, or $\G$, then the argument in the proof of the lemma shows that there is up to conjugacy a unique copy of $\sl_2$ in $\g$ with index 2, 2, or 3 respectively.  For $\g$ of type $\B_n$ with $n \ge 4$, there are two conjugacy classes of $\sl_2$-subalgebras of index 2.  This amounts to the fact that there are vectors in the $\mathrm{C}_n$ root lattice that are not roots but have the same length as a root---specifically, sums of two strongly orthogonal short roots, cf.~Exercise 5 in \S12 of \cite{Hum:LA}.

%%%%%%%%%%%%%%%%%%%%%%%%%%%%%%%%%%%%%%%%%%%%%%%%%%%%%%%%%%%%%%%%%
\section{Copies of $\sl_{2,\C}$ in the complex $\E_8$} \label{sl2.e8.sec}

We now prove some facts about copies of $\sl_{2,\C}$ in the complex Lie algebra $\e_8$ of type $\E_8$.  Of course, the 69 conjugacy classes of such are known---see \cite[pp.~182--185]{Dynk:ssub} or \cite[pp.~430--433]{Carter:big}---but we do not need this information.

Fix a pinning for $\e_8$; this includes a Cartan subalgebra $\h$, a set of simple roots $\Delta := \{ \alpha_i \mid 1 \le i \le 8 \}$ (numbered 
\begin{equation} \label{e8.num}
\begin{tabular}{ccccccc}
1&3&4&5&6&7&8\\
&&2
\end{tabular}
\end{equation}
as in \cite{Bou:g4}), and fundamental weights $\omega_i$ dual to $\alpha_i$. As all roots of the $\E_8$ root system have the same length, we can and do identify the root system with its  coroot system (also called the ``inverse" or ``dual" root system).

\begin{eg} \label{sl2.eg1}
Taking any root of $\E_8$, one can generate a copy of $\sl_{2,\C}$ in $\e_8$ with index 1.  Doing this with the highest root gives an $\sl_{2,\C}$ with marked Dynkin diagram
\[
\begin{tabular}{cc}
index 1:&
\begin{tabular}{ccccccc}
0&0&0&0&0&0&1\\
&&0
\end{tabular} 
\end{tabular}
\]
Every index 1 copy of $\sl_2$ in $\e_8$ is conjugate to this one by Lemma \ref{conj1}.
\end{eg}

\begin{eg} \label{sl2.eg2}
One can find a copy of $\sl_{2,\C} \times \sl_{2,\C}$ in $\e_8$ by taking the  first copy to be generated by the highest root of $\E_8$ and the second copy to be generated by the highest root of the obvious $\E_7$ subsystem.  If you embed $\sl_{2,\C}$ diagonally in this algebra, you find a copy of $\sl_{2,\C}$ with index 2 and marked Dynkin diagram
\[
\begin{tabular}{cc}
index 2:&
\begin{tabular}{ccccccc}
1&0&0&0&0&0&0\\
&&0
\end{tabular}
\end{tabular}
\]
\end{eg}

\begin{prop} \label{sl2}
The following collections of copies of $\sl_{2,\C}$ in $\e_8$ are the same:
\begin{enumerate}
\item copies such that $\pm 1$ are weights of $\e_8$ (as a representation of $\sl_{2,\C}$) and no other odd weights occur.
\item copies such that every weight of $\e_8$ is in $\{ 0, \pm 1, \pm 2 \}$.
\item copies such that the inclusion $\sl_{2,\C} \subset \e_8$ has Dynkin index $1$ or $2$.
\item copies of $\sl_{2,\C}$ conjugate to one of those defined in Examples \ref{sl2.eg1} or \ref{sl2.eg2}.
\end{enumerate}
\end{prop}

\begin{proof} %[Proof of Proposition \ref{sl2}]
One easily checks that (4) is contained in (1)--(3); we prove the opposite inclusion.

For (3), we identify $\h$ with the complexification $Q \ot \C$ of the (co)root lattice $Q$, hence $h$ with $\sum \alpha_i(h) \omega_i$.  By \eqref{index.alt}, the index of $h$ satisfies:
\[
\frac{1}{60} \sum_\alpha \alpha(h)^2 = \frac{1}{60} \sum_\alpha \left(\sum_i \alpha_i(h) \qform{\omega_i, \alpha}\right)^2 \ge \sum_i \left( \alpha_i(h)^2 \sum_\alpha \frac{\qform{\omega_i, \alpha}^2}{60} \right)
\]
where the sums vary over the positive roots.
We calculate for each fundamental weight $\omega_i$ the number $\sum_\alpha \qform{\omega_i, \alpha}^2/60$:
\begin{equation} \label{e8.indexes}
\begin{tabular}{rrrrrrr}
2&7&15&10&6&3&1 \\
&&4
\end{tabular}
\end{equation}
As the numbers $\alpha_i(h)$ are all 0, 1, or 2, the numbers \eqref{e8.indexes} show that $h$ for an $\sl_{2,\C}$ with Dynkin index 1 or 2 must be $\omega_1$ (index 2) or $\omega_8$ (index 1).

For (2), the highest root $\at$ of $\E_8$ is $\at = \sum_i c_i \alpha_i$, where $c_1 = c_8 = 2$ and the other $c_i$'s are all at least 3.  As $\at(h)$ is a weight of $\e_8$ relative to a given copy of $\sl_{2,\C}$, we deduce that an $\sl_{2,\C}$ as in (2) must have $h = \omega_1$ or $\omega_8$, as claimed.

\medskip
Suppose now that we are given an $h$ for a copy of $\sl_{2,\C}$ as in (1).  As $\pm 1$ occur as weights, there is at least one 1 in the marked Dynkin diagram.

But note that there cannot be three or more 1's in the marked Dynkin diagram for $h$.  Indeed, for every connected subset $S$ of vertices of the Dynkin diagram of $\E_8$, $\sum_{i \in S} \alpha_i$ is a root \cite[\S{VI.1.6}, Cor.~3b]{Bou:g4}.  If the number of 1's in the marked diagram of $h$ is at least three, then one can pick $S$ so that it meets exactly three of the $\alpha_i$'s with $\alpha_i(h) = 1$, in which case $\sum_{i \in S} \alpha_i(h)$ is odd and at least 3, violating the hypothesis of (1).

For sake of contradiction, suppose that there are two 1's in the marked diagram for $h$, say, corresponding to simple roots $\alpha_i$ and $\alpha_j$ with $i < j$.  For each $i, j$, one can find a root $\beta$ in the list of roots of $\E_8$ of large height in \cite[Plate VII]{Bou:g4} such that the coefficients of $\alpha_i$ and $\alpha_j$ in $\beta$ have opposite parity and sum at least 3.  (Merely taking $\beta$ to be the highest root suffices for many $(i, j)$.)  This contradicts (1), so there is a unique 1 in the marked diagram for $h$, i.e., $\alpha_i(h) = 1$ for a unique $i$.

If $\alpha_i(h) = 1$ for some $i \ne 1, 8$, then we find a contradiction because there is a root $\alpha$ of $\E_8$ with $\alpha_i$-coordinate 3.  Therefore $\alpha_i(h) = 1$ only for $i = 1$ or 8 and not for both.  By the fact used two paragraphs above, $\beta := \sum_i \alpha_i$ is a root of $\E_8$, so $\beta(h) = \sum \alpha_i(h)$ is odd and must be 1.  It follows that $h = \omega_1$ or $\omega_8$.
\end{proof}

\begin{borel}{Centralizer for index 1} \label{eg1.cent}
The $\sl_{2,\C}$ of index 1 in $\e_8$ has centralizer the obvious regular subalgebra $\e_7$ of type $\E_7$.  (A subalgebra is \emph{regular} if it is generated by the root subalgebras corresponding to a closed sub-root-system \cite[no.~16]{Dynk:ssub}.) Indeed, it is clear that $\e_7$ centralizes this $\sl_{2,\C}$.  Conversely, the centralizer of $\sl_{2,\C}$ is contained in the centralizer of $h = \omega_8$---i.e., $\e_7 \oplus \C h$---but does not contain $h$.
\end{borel}

\begin{borel}{Decomposing $\e_8$} \label{sl2.decomp}
Suppose we are given a copy of $\sl_{2,\C}$ in $\e_8$ specified by a defining vector $h$.  By applying the 240 roots of $\e_8$ to $h$ (and throwing in also 0 with multiplicity 8), we obtain the weights of $\e_8$ as a representation of $\sl_{2,\C}$ and therefore also the decomposition of $\e_8$ into irreducible representations of $\sl_{2,\C}$ as in, e.g., \cite[\S7.2]{Hum:LA}.

Extending this, suppose we are given a copy of $\sl_{2,\C} \times \sl_{2,\C}$ in $\e_8$, where the two summands are specified by defining vectors in $\h$.  (Here we want the defining vectors to span the Cartan subalgebras in the images of the two $\sl_{2,\C}$'s.  In particular, they need not be normalized in the sense of \eqref{dynk.norm}.)  Computing as in the previous paragraph, we can decompose $\e_8$ as a direct sum of irreducible representations $m \ot n$ of $\sl_{2,\C} \times \sl_{2,\C}$.  It is easy to write code from scratch to make a computer algebra system perform this computation.  We remark that applying this recipe in the situation from the introduction gives the dimension of $V_{m,n}$ as the multiplicity of $m \ot n$.
\end{borel}

%%%%%%%%%%%%%%%%%%%%%%%%%%%%%%%%%%%%%%%%%%%%%%%%%%%%%%%%%%%%%%%%%
\section{Index 2 copies of $\sl_{2,\C}$ in the complex $\E_8$} \label{ind2.sec}

\begin{lem} \label{eg2.cent}
The centralizer of the index 2 $\sl_{2,\C}$ in $\e_8$ from Example \ref{sl2.eg2}  is a copy of $\so_{13}$ contained in the regular subalgebra $\so_{14}$ of $\e_8$.
\end{lem}

\begin{proof}
The centralizer of the $\sl_{2,\C}$ of index 2 in $\e_8$ is contained in the centralizer of the defining vector $h$; this centralizer is reductive with semisimple part the regular subalgebra $\so_{14}$ of type $\D_7$.  The centralizer of $\sl_{2,\C}$ contains the centralizer of the $\sl_{2,\C} \times \sl_{2,\C}$ from Example \ref{sl2.eg2}, which is the regular subalgebra $\so_{12}$ of type $\D_6$, as can be seen by the recipe from \cite[pp.~147, 148]{Dynk:ssub}.
Computing as in \ref{sl2.decomp}, we see that the centralizer of $\sl_{2,\C}$ has dimension 78 (as is implicitly claimed in the statement of the lemma), so it lies properly between the regular $\so_{12}$ and the regular $\so_{14}$.  

For concreteness, let us suppose that the structure constants for $\e_8$ are as in \cite{Vavilov:const}.  Define a copy of $\sl_{2,\C}$ by sending $\stbtmat{0}{1}{0}{0}$ to the sum of the elements in the Chevalley basis of $\e_8$ spanning the root subalgebras corresponding to $-\alpha_8$ and the highest root in the obvious $\D_7$ subdiagram.  This copy of $\sl_{2,\C}$ has defining vector $\alpha_2 + \alpha_3 + 2\alpha_4 + 2\alpha_5 + 2\alpha_6 + 2\alpha_7$.  One checks using the structure constants that this $\sl_{2,\C}$ centralizes the index 2 $\sl_{2,\C}$ we started with, and that together with $\so_{12}$ it generates a copy of $\so_{13}$.  In particular, the coroot lattice of this $\so_{13}$ has basis $\bch_1, \ldots, \bch_6$, embedded in the (co)root lattice of $\e_8$ as in the table:
\begin{equation} \label{b6e8}
\begin{array}{c|cccccc}
\so_{13}&\bch_1&\bch_2&\bch_3&\bch_4&\bch_5&\bch_6 \\ \hline
\e_8&\alpha_3&\alpha_4&\alpha_5&\alpha_6&\alpha_7&-\alpha_2- \alpha_3-2\alpha_4 - 2\alpha_5 - 2\alpha_6 - 2\alpha_7
\end{array}
\end{equation}
We remark that the numbering of the coroots $\bch_1, \ldots, \bch_6$ corresponds
to a numbering of the simple roots of $\so_{13}$ as in the diagram
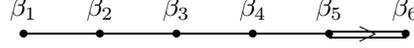
\[
\begin{picture}(7,0.5)
    % put in the horizontal lines
    \put(0.25,0.10){\line(1,0){4}} %multiput(1,1)(1,0){2}{\line(1,0){1}}
    \put(4.25,0.13){\line(1,0){1}}
    \put(4.25,0.05){\line(1,0){1}}
 
    % label horizontal circles
    \put(0.25,0.27){\makebox(0,0.4)[b]{$\beta_1$}}
    \put(1.25,0.27){\makebox(0,0.4)[b]{$\beta_2$}}
    \put(2.25,0.27){\makebox(0,0.4)[b]{$\beta_3$}}
    \put(3.25,0.27){\makebox(0,0.4)[b]{$\beta_4$}}
    \put(4.25,0.27){\makebox(0,0.4)[b]{$\beta_5$}}
    \put(5.25, 0.27){\makebox(0,0.4)[b]{$\beta_6$}}
    
    \put(4.75,0){\makebox(0,0.4)[b]{$>$}}

    % put in the horizontal circles
    \multiput(0.25,0.1)(1,0){6}{\circle*{\darkrad}}

%    % put in the diagonal circles
%    \put(5.7,0.3){\circle{0.2}}
%    \put(5.7,1.7){\circle{0.2}}
%    \put(5.7,-0.1){\makebox(0,0.4)[b]{$\ell^{-}$}}
%    \put(5.7,1.9){\makebox(0,0.4)[b]{$\ell^{+}$}}
\end{picture}
\]
Dimension count shows that this $\so_{13}$ is the centralizer.
\end{proof}

%\begin{borel*} \label{so13}
The claim of the lemma is already in \cite[p.~125]{Elashvili:Z}.  We gave the details of a proof because it specifies an inclusion of $\so_{13}$ in $\e_8$ and a comparison of the pinnings of the two algebras as in \eqref{b6e8}.  

The index 2 $\sl_2$ and the copy of $\so_{13}$ give an $\sl_{2} \times \so_{13}$ subalgebra of $\e_{8}$.  We now decompose $\e_8$ into irreducible representations of $\sl_2 \times \so_{13}$.  We can do this from first principles by restricting the roots of $\e_8$ to the Cartan sublagebras of $\sl_2$ (using the marked Dynkin diagram from Example \ref{sl2.eg2}) and $\so_{13}$ (using \eqref{b6e8}).  Alternatively, we can read the decomposition off the tables in \cite{McKP} as follows.  As in the proof of Lemma \ref{eg2.cent}, $\sl_2$ is contained in the regular subalgebra $\sl_2 \times \sl_2 \times \so_{12}$ of $\e_8$, and the tables on pages 301 and 305 of ibid.\ show that $\e_8$ decomposes as a sum of 
\begin{equation} \label{so12}
\text{the adjoint representation,} \quad 2 \ot 1 \ot  S_+,\quad 1 \ot 2 \ot S_-, \quad \text{and} \quad 2 \ot 2 \ot V,
\end{equation}
where $S_\pm$ denotes the half-spin representations of $\so_{12}$ and $V$ is the vector representation.  We can restrict the representations of $\sl_2 \times \sl_2$ to the diagonal $\sl_2$ subalgebra to obtain a decomposition of $\e_8$ into representations of $\sl_2 \times \so_{12}$.  Consulting the tables in ibid.\ for restricting representations from type $\B_6$ to $\D_6$ allows us to deduce the decomposition
\begin{equation} \label{so13}
1 \ot \so_{13,\C} \quad \oplus \quad 2 \ot (\text{spin}) \quad \oplus \quad 3 \ot 1 \quad \oplus \quad 3 \ot (\text{vector})
\end{equation}
of $\e_8$ as a representation of $\sl_2 \times \so_{13}$.
From this it is obvious that $\so_{13,\C}$ is the Lie algebra of a copy of $\Spin_{13}$ in $\E_8$.
%\end{borel*}

%\begin{eg} \label{sl2.eg2p}
%By Dynkin's game of adding the highest root to the Dynkin diagram and deleting a vertex as in \cite[\S5]{Dynk:ssub}, $\so_{13,\C}$ contains a maximal subalgebra $\so_{8,\C} \times \sp_{4,\C}$.  In turn, $\so_8$ contains a maximal subalgebra $\sl_2 \times \sp_4$ \cite[Th.~1.4]{Dynk:max}.  This gives an $\sl_2 \times \sp_4 \times \sp_4$ subalgebra of $\so_{13}$.

%We remark that this $\sl_2$ has index 2 in $\e_8$.  As the inclusions $\so_8 \subset \so_{13} \subset \e_8$ have index 1, it suffices to check that $\sl_2$ has index 2 in $\so_8$.  This follows from the fact that the adjoint representation of $\so_8$ has index 12, whereas its restriction to $\sl_2$ decomposes as six copies of the 3-dimensional irreducible representation and a 10-dimensional trivial representation \cite[p.~260]{McKP}, so has index $6 \cdot 4 + 10 \cdot 0 = 24$.
%\end{eg}

The main result of this section is the following:
\begin{lem} \label{ind22}
Up to conjugacy, there is a unique copy of $\SL_{2,\C} \times \SL_{2,\C}$ in $\E_{8,\C}$ so that each inclusion of $\SL_{2,\C}$ in $\E_{8,\C}$ has index 2.  The centralizer of this $\SL_{2,\C} \times \SL_{2,\C}$ has identity component $\Sp_{4,\C} \times \Sp_{4,\C}$.
\end{lem}

\begin{proof}
As in the proof of Lemma \ref{conj1} (or by the method used to prove Prop.~\ref{sl2}), there are two index 2 copies of $\sl_2$ in $\so_{13}$, coresponding to the partitions
\[
\text{(a)}\quad 3+1+1+\cdots+1 
\quad \text{and} \quad
\text{(b)} \quad 2+2+2+2+1+1+\cdots+1
\]
of 13.  The recipe in \cite[\S5.3]{CMcG} gives defining vectors for these $\sl_2$'s, which we can rewrite in terms of the $\E_8$ simple roots using \eqref{b6e8}:
\begin{equation} \label{ind22.2}
\begin{array}{cl}
\text{(a)}&2\bch_1+2\bch_2+2\bch_3+2\bch_4+2\bch_5+\bch_6 = -\alpha_2+\alpha_3\\
\text{(b)}&\bch_1+2\bch_2+3\bch_3+4\bch_4+4\bch_5+2\bch_6 = -2\alpha_2-\alpha_3-2\alpha_4-\alpha_5
\end{array}
\end{equation}
We can pair each of (a) and (b) with the copy of $\sl_2$ from Example \ref{sl2.eg2} to get an $\sl_2 \times \sl_2$ subalgebra of $\e_8$ where both $\sl_2$'s have index 2.  Clearly, these represent the only two $\E_8$-conjugacy classes of such subalgebras.  With \eqref{ind22.2} in hand, we can calculate the multiplicities of the irreducible representations of $\sl_2 \times \sl_2$ in $\e_8$ as in \ref{sl2.decomp}.

In case (a), every irreducible summand $m \ot n$ has $m + n$ even.  Therefore, this copy of $\sl_2 \times \sl_2$ is the Lie algebra of a subgroup of $\E_8$ isomorphic to $(\SL_2 \times \SL_2)/(-1, -1)$.  (An alternative way to see this is to note that the simple roots with odd coefficients are the same in (\ref{ind22.2}a) and the defining vector in Example \ref{sl2.eg2}.)

In case (b), we have the following table of multiplicities for $m \ot n$:
\begin{equation} \label{22.table}
\begin{array}{cc|rrrc} 
&&1&2&3&m\\ \hline
&1&20&20&6 \\
n&2&20&16&4 \\
&3&6&4&0 
\end{array}
\end{equation}
In particular, it is the Lie algebra of a copy of $\SL_2 \times \SL_2$ in $\E_8$.  The centralizer of (b) in $\Spin_{13}$ has been calculated in \cite[IV.2.25]{SpSt}, and the identity component is $\Sp_4 \times \Sp_4$, as claimed.
\end{proof}

%\begin{borel*} \label{decomp.22}
We can decompose $\e_8$ into a direct sum of irreducible representations of the $\sl_2 \times \sl_2 \times \sp_4 \times \sp_4$ subalgebra from Lemma \ref{ind22} by combining 
the decomposition of $\e_8$ into irreducible representations of $\sl_2 \times \so_{13}$ from \eqref{so13} with the tables in \cite{McKP}.   Specifically, we restrict representations from $\so_{13}$ to an $\sp_4 \times \so_8$ subalgebra and then from $\so_8$ to $\sp_4 \times \sl_2$, where this $\sl_2$ also has index 2.
Recall that $\sp_4$ has two fundamental irreducible representations: one that is 4-dimensional symplectic and another that is 5-dimensional orthogonal; we denote them by their dimensions.  With this notation and \ref{decomp}, we find:
\begin{equation} \label{decomp.22}
V_{2,1} \simeq 5 \ot 4, \quad V_{1,2} \simeq 4 \ot 5, \quad V_{2,3} \simeq 1 \ot 4, \quad V_{3,2} \simeq 4 \ot 1, \eand V_{2,2} \simeq 4 \ot 4.
\end{equation}
%\end{borel*}

%%%%%%%%%%%%%%%%%%%%%%%%%%%%%%%%%%%%%%%%%%%%%%%%%%%%%%%%%%%%%%%%%
\section{Copies of $\SL(2, \C)$ in a real form of $\E_8$}

Suppose now that we have a copy of $\SL(2, \C)$ inside a real Lie group $\cE$ of type $\E_8$.  Over the complex numbers, we decompose $\Lie(\cE) \ot \C$ into a direct sum of irreducible representations of $\SL(2, \C) \ot \C \simeq \SL_{2,\C} \times \SL_{2,\C}$; each irreducible representation can be written as $m \ot n$ where $m$ and $n$ denote the dimension of an irreducible representation of the first or second $\SL_{2,\C}$ respectively.  The goal of this section is to prove:

\begin{prop} \label{real.z}
Maintain the notation of the previous paragraph.  If $\Lie(\cE) \ot \C$ contains no irreducible summands $m \ot n$ with $m+n > 4$, then the identity component $Z$ of the centralizer of $\SL(2, \C)$ in $\E$ is a subgroup isomorphic to
\begin{enumerate}
\item  $\Spin(7,5)$ if $\cE$ is split; or 
\item $\Spin(9,3)$ or $\Spin(11, 1)$ if the Killing form of $\Lie(\cE)$ has signature $-24$.
\end{enumerate}
In either case, $\Lie(Z) \ot \C$ is the regular $\so_{12}$ subalgebra of $\Lie(\cE) \ot \C$.
\end{prop}

\begin{proof}
Complexifying the inclusion of $\SL(2, \C)$ in $\cE$ and going to Lie algebras gives an inclusion of $\sl_{2,\C} \times \sl_{2,\C}$ in the complex Lie algebra $\e_8$.  The hypothesis on the irreducible summands $m \ot n$ implies that each of the two $\sl_{2,\C}$'s has index 1 or 2 by Proposition \ref{sl2}.  As complex conjugation interchanges the two components, they must have the same index.  

Suppose first that both $\sl_2$'s have index 2.  When we decompose $\e_8$ as in \ref{decomp}, we find the representation $2 \ot 3$ with positive multiplicity 4 by \eqref{22.table}, which violates our hypothesis on the $\SL(2, \C)$ subgroup of $\cE$.

Therefore both $\sl_2$'s have index 1.
Lemma \ref{conj1} (twice) gives that this $\sl_2 \times \sl_2$ is conjugate to the one generated by the highest root of $\E_8$ from Example \ref{sl2.eg1} (so the second $\sl_2$ belongs to the centralizer of type $\E_7$) and by the highest root of the $\E_7$ subsystem and makes up the first two summands of an $\sl_2 \times \sl_2 \times \so_{12}$ subalgebra, the same one used to find \eqref{so12}.  That is, $\so_{12}$ centralizes $\sl_2 \times \sl_2$.  Conversely, the centralizer of the defining vectors of the two copies of $\sl_2$ has semisimple part $\so_{12}$; it follows that $\Lie(Z) \ot \C$ is isomorphic to $\so_{12}$.

From this and the decomposition \eqref{so12}, we see that $Z$ is a real form of $\Spin_{12}$.  As $\Lie(\cE)$ is a real representation of $Z$, we deduce that $V$ is also a real representation of $Z$ but $S_+$ and $S_-$ are not; they are interchanged by the Galois action.  The first observation shows that $Z$ is $\Spin(12-a, a)$ for some $0 \le a \le 6$.  The second shows that $a$ must be 1, 3, or 5, as claimed in the statement of the proposition.

It remains to prove the correspondence between $a$ and the real forms of $\E_8$.  For $a = 5$, this is clear: the subgroup generated by $\SL(2, \C)$ and $\Spin(7,5)$ has real rank 6, so it can only be contained in the split real form. 

Now suppose that $a = 3$ or 1 and that $\SL(2, \C)$ is in the split $\E_8$; we will show that the Killing form of $\cE$ has signature $-24$.  Over $\C$, $\SL(2, \C)$ is conjugate to the copy of $\SL_{2,\C} \times \SL_{2,\C}$ in $\E_{8,\C}$ generated by the highest root of $\E_8$ and the highest root of the natural subsystem of type $\E_7$.  
Writing out these two roots in terms of the $\E_8$ simple roots, we see that $\alpha_3$ and $\alpha_5$ are the only simple roots whose coefficients have different parities.  It follows that the element $-1 \in \SL(2, \C)$---equivalently, $(-1, -1) \in \SL_2 \times \SL_2$---is $h_{\alpha_2}(-1)\,h_{\alpha_3}(-1)$ in the notation of \cite{St}, where $h_{\alpha_i} \!: \C^\times \ra \cE \ot \C$ is the cocharacter corresponding to the coroot $\ach_i$.  Now, $\alpha_2$ and $\alpha_3$ are the only simple roots with odd coefficients in the fundamental weight $\omega_1$, so the subgroup of $\cE \ot \C$ fixed by conjugation by this $-1$ is generated by root subgroups corresponding to roots $\alpha$ such that $\qform{\omega_1, \alpha}$ is even.  These roots form the natural $\D_8$ subsystem of $\E_8$, and in this way we see $\SL(2, \C) \cdot \Spin(12 - a, a)$ as a semisimple subgroup of maximal rank in a copy of a half-spin group $H$ in 16 dimensions---the identity component of the centralizer of $-1$.

We claim that $H$ is isogenous to $\SO(12,4)$.  As $H$ is a half-spin group with a half-spin representation defined over $\R$, it is isogenous to $\SO(16-b, b)$ for $b = 0$, $4$, or $8$ or it is quaternionic; these possibilities have Killing forms of signature $-120$, $-24$, $8$, or $-8$ respectively, as can be looked up in \cite{Ti:tab}, for example.  The adjoint representation of $H$, when restricted to $\SL(2,\C) \cdot \Spin(12-a,a)$, decomposes as the adjoint representation of $\SL(2, \C) \cdot \Spin(12-a,a)$ and $2 \ot 2 \ot V$ by \eqref{so12}.  The Killing form on $H$ restricts to a positive multiple of the Killing form on $\SL(2, \C) \cdot \Spin(12-a,a)$ (as can be seen over $\C$ by the explicit formula on p.~E-14 of \cite{SpSt})---i.e., has signature $-44$ or $-12$ for $a = 1$ or 3---and a form of signature $\pm 2 (12-2a)$ on $2 \ot 2 \ot V$; the sum of these has signature $0$, $-24$, or $-64$ since $a = 1$ or $3$.  Comparing the two lists verifies that $H$ is isogenous to $\SO(12,4)$.

The Killing form on $H$ has signature $-24$.  The invariant bilinear form on the half-spin representation is hyperbolic (because $H$ is isogenous to spin of an isotropic quadratic form of dimension divisible by 8, see \cite[1.1]{G:clif}).  As a representation of $H$, $\Lie(E)$ is a sum of these two representations, and we conclude that the Killing form on $\Lie(E)$ has signature $-24$, as claimed.
\end{proof}

\begin{rmk} \label{real.toep}
We can determine the centralizer and the real form of $\E_8$ also in the excluded case in the proof where both $\sl_2$'s have index 2.  As in Lemma \ref{ind22}, the centralizer is a real form of $\Sp_{4,\C} \times \Sp_{4,\C}$.
The decomposition \eqref{decomp.22} shows that complex conjugation interchanges the two $\Sp_{4,\C}$ terms, so the centralizer is $R(\Sp_{4,\C})$.  Complex conjugation interchanges the irreducible representations appearing in \eqref{decomp} in pairs (contributing 0 to the signature of the Killing form $\kappa_E$ of $E$), except for $2 \ot 2 \ot V_{2,2}$, which has dimension $8^2$.  This last piece breaks up into a 36-dimensional even subspace, and a 28-dimensional odd subspace, contributing 8 to the signature of $\kappa_E$ and proving that the resulting real form of $\E_8$ is the split one.
\end{rmk}

%%%%%%%%%%%%%%%%%%%%%%%%%%%%%%%%%%%%%%%%%%%%%%%%%%%%%%%%%%%%%%%%%
\section{No Theory of Everything in a real form of $\E_8$} \label{real.phys.sec}

In the decomposition \eqref{decomp} of $\Lie(\cE) \ot \C$, the integers $m, n$ are positive, so \ToE2 implies
\begin{equation}
\text{$V_{m,n} = 0$ if $m \ge 4$ or $n \ge 4$.} \tag{ToE2'}\label{ToE2'}
\end{equation}
We prove the following strengthening of the real case of Theorem \ref{MT}:

\begin{lem} \label{MT.R}
If subgroups $\SL(2,\C)$ and $G$ of a real form $\cE$ of $\E_8$ satisfy \ToE1 and \ToEp, then $V_{1,2}$ is a self-conjugate representation of $G$, i.e., \ToE3 fails.
\end{lem}

\begin{proof}
As in the proof of Proposition \ref{real.z}, over the complex numbers we get two copies of $\sl_2$ that embed in $\E_8$ with the same index, which is 1 or 2.  

If the index is 1, we are in the case of that proposition.
The $-1$-eigenspace in $\Lie(\cE)$ (of the element $-1$ in the center of $\SL(2, \C)$) is a real representation of $\SL(2, \C)\cdot G$, and $G$ is contained in a copy of $\Spin(12-a,a)$ for $a = 1$, 3, or 5.  As in the proof of the proposition, there is a representation $W$ of $\SL(2, \C) \times \Spin(12-a,a)$ defined over $\R$ that is isomorphic to 
\[
(2 \ot 1 \ot  S_+) \quad \oplus \quad (1 \ot 2 \ot S_-)
\]
over $\C$.  Now $G$ is contained in the maximal compact subgroup of $\Spin(12-a,a)$, i.e., $\Lie(G)$ is a subalgebra of $\so(11)$, $\so(9) \times \so(3)$, or $\so(7) \times \so(5)$.  The restriction of the two half-spin representations of $\Spin(12-a, a)$ to the compact subalgebra are equivalent \cite[p.~264]{McKP}, and we see that in each case the restriction is \emph{quaternionic}.  (To see this, one uses the standard fact that the spin representation of $\so(2\ell + 1)$ is real for $\ell \equiv 0, 3 \pmod{4}$ and quaternionic for $\ell \equiv 1, 2 \pmod{4}$.)  That is, the restrictions of $S_+$, $S_-$, and their complex conjugates to the maximal compact subgroup are all equivalent (over $\C$), hence the same is true for their further restrictions to $G$, and \ToE3 fails.

If the index is 2, then $G$ is contained in a real form of $\Sp_{4,\C} \times \Sp_{4,\C}$ by Lemma \ref{ind22}.  When we decompose $\e_8$ as in \eqref{decomp}, we find $V_{2,1}$ and $V_{1,2}$ as in \eqref{decomp.22}.  As complex conjugation interchanges these two representations, it follows that complex conjugation interchanges the two $\Sp_{4,\C}$ factors, i.e., the centralizer of $\SL(2, \C)$ has identity component the transfer $R(\Sp_{4,\C})$ of $\Sp_{4,\C}$.  Its maximal compact subgroup is the compact form of $\Sp_{4,\C}$ (also known as $\Spin(5)$), all of whose irreducible representations are self-conjugate.  Therefore, \ToE3 fails.
\end{proof}

\begin{rmk}
It is worthwhile noting that, in each of the three cases in Proposition \ref{real.z} (the three cases where \ToE2 holds), it is possible to embed $G_\SM$ in the centralizer, thus showing that \ToE1 is satisfied. Given such an embedding, a simple computation verifies explicitly that $S_+$ has a self-conjugate structure as a representation of $G_\SM$.

First consider $\Spin(11,1)$. There is an obvious embedding of $G_\GUT := \Spin(10)$. Under this embedding, $S_+$ decomposes as the direct sum of the two half-spinor representations, i.e., as a generation and an anti-generation.

For $\Spin(7,5)$, there is an obvious embedding of the Pati-Salam group, $G_\GUT := (\Spin(6)\times \Spin(4))/ (\Zm2)$. Again, $S_+$ decomposes as the direct sum of a generation and an anti-generation.

Finally, $\Spin(3,9)$ contains $(\SU(3)\times \SU(2)\times \SU(2)\times \U(1))/(\Zm6)$ as a subgroup. Under this subgroup,
\[
  S_+ = (3,2,2)_{1/6} \oplus (\overline{3},2,2)_{-1/6} + (1,2,2)_{-1/2} + (1,2,2)_{1/2}
\]
where the subscript indicates the $\U(1)$ weights, and the overall normalization is chosen to agree with the physicists' convention for the weights of the Standard Model's $\U(1)_Y$. Embedding the $\SU(2)$ of the Standard Model in one of the two $\SU(2)$s, we obtain an embedding of $G_\SM\subset \Spin(3,9)$ where, again $S_+$ has a self-conjugate structure as a representation of $G_\SM$.
\end{rmk}

%%%%%%%%%%%%%%%%%%%%%%%%%%%%%%%%%%%%%%%%%%%%%%%%%%%%%%%%%%%%%%%%%
\section{No Theory of Everything in complex $\E_8$} \label{c.phys.sec}

We now complete the proof of Theorem \ref{MT} by proving the following strengthening of the complex case.
\begin{lem} \label{MT.C}
If subgroups $\SL(2,\C)$ and $G$ of $R(E_{8,\C})$ satisfy \ToE1 and \ToEp, then $V_{1,2}$ is a self-conjugate representation of $G$, i.e., \ToE3 fails.
\end{lem} 

First, recall the definition of the transfer $R(H_\C)$ of a complex group $H_\C$ as described, e.g., in \cite[\S2.1.2]{PlatRap}.  Its complexification can be viewed as $H_\C \times H_\C$, where complex conjugation acts via
\[ %begin{equation} \label{conj}
\overline{(h_1, h_2)} = (\overline{h_2},\, \overline{h_1}).
\] %end{equation}
One can view $R(H_\C)$ as the subgroup of the complexification consisting of elements fixed by complex conjugation.

Now consider an inclusion $\phi \!: \SL(2, \C) = R(\SL_{2,\C}) \injects R(\E_{8,\C})$.  Complexifying, we identify $R(\SL_{2,\C})\ot \C$ with $\SL_{2,\C} \times \SL_{2,\C}$ and similarly for $R(E_{8,\C})$ and write out $\phi$ as
\begin{equation} \label{phi.img}
\phi(h_1, h_2) = (\phi_1(h_1) \phi_2(h_2), \psi_1(h_1) \psi_2(h_2))
\end{equation}
for some homomorphisms $\phi_1, \phi_2, \psi_1, \psi_2 \!: \SL_{2,\C} \ra \E_{8,\C}$.  As $\phi$ is defined over $\R$, we have:
\[
\phi(h_1, h_2) 
= \overline{\phi(\overline{h_2}, \overline{h_1})}  =
(\overline{\psi_1(\overline{h_2}) \psi_2(\overline{h_1})}, \overline{\phi_1(\overline{h_2}) \phi_2(\overline{h_1})}),
\]
and it follows that $\psi_1(h_1) = \overline{\phi_2(\overline{h_1})}$ and $\psi_2(h_2) = \overline{\phi_1(\overline{h_2})}$.  Conversely, given any two homomorphisms $\phi_1, \phi_2 \!: \SL_{2,\C} \ra \E_{8,\C}$ (over $\C$) with commuting images, the same equations define a homomorphism $\phi \!: \SL(2, \C) \ra R(\E_{8,\C})$ defined over $\R$.

\begin{proof}[Proof of Lemma \ref{MT.C}]
Write $Z$ for the identity component of the centralizer of the image of the map $\phi_1 \times \phi_2 \!: \SL_{2,\C} \times \SL_{2,\C} \ra \E_{8,\C}$ from \eqref{phi.img}.  Clearly, $G$ is contained in the transfer $R(Z)$ of $Z$.  In each of the cases below, we verify that
\begin{equation} \label{complex.ss}
\parbox{4in}{$Z$ is semisimple and $-1$ is in the Weyl group of $Z$.}
\end{equation}
It follows from this that the maximal compact subgroup of $R(Z)$ is the compact real form $Z_\R$ of $Z$ and that $Z_\R$ is an \emph{inner} form.
Hence every irreducible representation of $Z_\R$ is real or quaternionic, hence every representation of $Z_\R$ is self-conjugate.  That is, \ToE3 fails, which is the desired contradiction.

\smallskip
{\underline{\textsl{Case 1: $\phi_1$ or $\phi_2$ is trivial.}}}
Consider the easiest-to-understand case where $\phi_1$ or $\phi_2$ is the zero map, say $\phi_2$.  In the notation of \eqref{phi.img}, $\phi(h_1, h_2) = (\phi_1(h_1), \overline{\phi_1(\overline{h_2})})$, i.e., $\phi$ is the transfer of the homomorphism $\phi_1 \!: \SL_{2,\C} \ra \E_{8,\C}$.  By Proposition \ref{sl2}, $\phi_1$ has index 1 or 2.
If $\phi_1$ has index 1, then $Z$ is simple of type $\E_7$ by \ref{eg1.cent}, hence \eqref{complex.ss} holds.  If $\phi_1$ has index 2, then $\Lie(Z)$ is isomorphic to $\so_{13,\C}$ by Lemma \ref{eg2.cent}, and again \eqref{complex.ss} holds.
%\end{borel}

\smallskip
{\underline{\textsl{Case 2: Neither $\phi_1$ nor $\phi_2$ is trivial.}}}
Now suppose that neither $\phi_1$ nor $\phi_2$ is trivial.  Again, Proposition \ref{sl2} implies that $\phi_1$ and $\phi_2$ have Dynkin index 1 or 2.

If $\phi_1$ and $\phi_2$ both have index 1, then (over $\C)$ the homomorphism $\phi_1 \times \phi_2$ is the one from the proof of Proposition \ref{real.z} and $Z$ is the standard $\D_6$ subgroup of $\E_{8,\C}$ and \eqref{complex.ss} holds.

Now suppose that $\phi_1$ and $\phi_2$ both have index 2.  As $\phi$ is an injection, it is not possible that $\phi_1$ and $\phi_2$ both vanish on $-1 \in \SL_{2,\C}$, and it follows from the proof of Lemma \ref{ind22} that $\phi_1 \times \phi_2$ is an injection as in the statement of Lemma \ref{ind22}.  In particular, $Z$ has Lie algebra $\sp_{4,\C} \times \sp_{4,\C}$ of type $\B_2 \times \B_2$ and \eqref{complex.ss} holds.  Note that \ToE2 fails in this case by \eqref{22.table}.

\smallskip
Suppose finally that $\phi_1$ has index 1 and $\phi_2$ has index 2.
We conjugate so that $\phi_2(\sl_2)$ is the copy of $\sl_2$ from Example \ref{sl2.eg2}, and (by Lemma \ref{conj1} for the centralizer $\so_{13}$ of $\phi_2(\sl_2)$) we can take $\phi_1(\sl_2)$ to be a copy of $\sl_2$ generated by the highest root of $\so_{13}$.  Calculating as described in \ref{sl2.decomp}  gives the following table of multiplicities for the irreducible representation $m \ot n$ of $\sl_2 \times \sl_2$ in $\e_8$:
\begin{equation} \label{12.table}
\begin{array}{cc|rrrc} 
&&1&2&3&m\\ \hline
&1&39&18&1 \\
n&2&32&16&0 \\
&3&10&2&0 
\end{array}
\end{equation}
In particular, the $\mathrm{A}_1 \times \B_4$ subgroup of $\Spin_{13}$ that centralizes the image of $\phi_1 \times \phi_2$ is all of the identity component $Z$ of the centralizer in $\E_8$.
Again \eqref{complex.ss} holds.  (Of course, \eqref{12.table} shows that (ToE2) fails.)
\end{proof}

%%%%%%%%%%%%%%%%%%%%%%%%%%%%%%%%%%%%%%%%%%%%%%%%%%%%%%%%%%%%%%%%%
\section{Relaxing \ToE2 to \ToEp} \label{supergravity}

Combining Lemmas \ref{MT.R} and \ref{MT.C} gives a proof not only of Theorem \ref{MT}, but of the following stronger statement.
\begin{thm} \label{MTp}
There are no subgroups $\SL(2, \C) \cdot G$ satisfying \ToE1, \ToEp, and \ToE3 in the (transfer of the) complex $\E_8$ or any real form of $\E_8$. $\hfill\qed$
\end{thm}

We retained hypothesis \ToE2 in the introduction because that is what is demanded by physics.
Technically, it is \emph{possible} for $V_{2,3}$ and $V_{3,2}$ to be nonzero in an interacting theory---so \ToE2 is false but \ToEp\ still holds---but only in the presence of local supersymmetry (i.e., in supergravity theories) \cite{Grisaru:1977kk}. Lisi's framework is not compatible with local supersymmetry, so we excluded this possibility above.

For real forms of $\E_8$, weakening \ToE2 to \ToEp\ only adds the case of $\E_{8(8)}$, with $G_{\text{max}}=\Spin(5)$, where we find
\begin{equation}\label{nonsugradecomp}
V_{3,2} \simeq V_{2,3} = 4, \qquad V_{2,1}\simeq V_{1,2} = 4 \oplus 16
\end{equation}
and we have indicated the irreducible representations of $\Spin(5)$ by their dimensions. Because the gravitinos transform nontrivially under $G_{\text{max}}$ and because of their multiplicity, the only consistent possibility would be a gauged $\mathcal{N}=4$ supergravity theory (for a recent review of such theories, see \cite{SchonWeidner}). Unfortunately, the rest of the matter content (it suffices to look at $V_{2,1}$) is not compatible with $\mathcal{N}=4$ supersymmetry. Even if it were, $\mathcal{N}=4$ supersymmetry would, of course, necessitate that the theory be non-chiral, making it unsuitable as a candidate Theory of Everything.

To summarize the results of this section, the previous subsection, and Remark \ref{real.toep}, weakening \ToE2 to \ToEp\ adds only three additional entries to Table \ref{Results}.
\begin{equation}\label{ExtraResults}
\begin{array}{cccc}
\cE & G_\mx & V_{3,2} & V_{2,1}\\ \hline
\E_{8(8)} & \Spin(5) & 4 & 4\oplus 16 \\
R(\E_{8,\C}) & \Spin(5)\times \Spin(5) & (4,1)\oplus (1,4) & (4,5)\oplus (5,4) \\
R(\E_{8,\C}) & \SU(2)\times \Spin(9) & (2,1) & (2,9)\oplus (2,16) \\
\end{array}
\end{equation}
In each case the fermion representations, $V_{2,1}\simeq V_{1,2}$ and $V_{3,2}\simeq V_{2,3}$, are pseudoreal representations of $G_\mx$.

%%%%%%%%%%%%%%%%%%%%%%%%%%%%%%%%%%%%%%%%%%%%%%%%%%%%%%%%%%%%%%%%%
\section{Conclusion}

In paragraph \ref{dim} above, we observed by an easy dimension count that no proposed Theory of Everything constructed using subgroups of a real form $\cE$ of $\E_8$ has a sufficient number of weight vectors in the $-1$-eigenspace to identify with all known fermions.  The proof of our Theorem  \ref{MT} was quite a bit more complicated, but it also gives much more.  It shows that you cannot obtain a \emph{chiral} gauge theory for \emph{any} candidate ToE subgroup of $\cE$, whether $\cE$ is a real form or the complex form of $\E_8$. In particular, it is impossible to obtain even the 1-generation Standard Model (in the sense of Definition \ref{ngen.def}) in this fashion.

\bigskip
\noindent{\small{\textbf{Acknowledgments.}} The second author thanks Fred Helenius for pointing out a reference and Patrick Brosnan for suggesting this project. The authors' research was supported by  the National Science Foundation under grant nos.\ PHY-0455649 (Distler) and DMS-0653502 (Garibaldi).}

\bibliographystyle{utphys}
\bibliography{physics}

\end{document}

%% file: skip-thm.tex
\usepackage{amsthm}

\newtheorem{thm}[equation]{Theorem}
\newtheorem{lem}[equation]{Lemma}

\newtheorem{prop}[equation]{Proposition}

\newtheorem*{thm*}{Theorem}
\newtheorem*{prop*}{Proposition}
\newtheorem*{cor*}{Corollary}
\newtheorem*{lem*}{Lemma}
\newtheorem*{MT*}{Main Theorem}

\theoremstyle{definition} %
\newtheorem{defn}[equation]{Definition}
\newtheorem*{defn*}{Definition}

\newtheorem{eg}[equation]{Example}

\theoremstyle{remark} %
\newtheorem{rmk}[equation]{Remark}

\newtheorem*{rmk*}{Remark}
\newtheorem*{rmks*}{Remarks}

\newtheoremstyle{exercise}% name
  {3pt}%      Space above
  {3pt}%      Space below
  {\small}%         Body font
  {}%         Indent amount (empty = no indent, \parindent = para indent)
  {\sc\small}% Thm head font
  {.}%        Punctuation after thm head
  {.5em}%     Space after thm head: " " = normal interword space;
   {}     %       \newline = linebreak
  {}%         Thm head spec (can be left empty, meaning `normal')

\theoremstyle{exercise}

\makeatletter
{\renewcommand{\theequation}{#1}}%
{\renewcommand{\theequation}{\arabic{equation}}\addtocounter{equation}{-1}\global\@ignoretrue}
\makeatother

\makeatletter
{\renewcommand{\theequation}{#1}\begin{eqnarray}}%
{\end{eqnarray}\renewcommand{\theequation}{\arabic{equation}}\addtocounter{equation}{-1}\global\@ignoretrue}
\makeatother

\makeatletter
\newenvironment{borel}[1]%
{\smallskip \refstepcounter{equation}\noindent{\textbf{\theequation.} }{{\textbf{#1.}}}}%
{\smallskip \global\@ignoretrue}
\makeatother

\makeatletter
{\smallskip \refstepcounter{equation}\noindent{\textbf{\theequation.} }{{\textbf{#1}}}}%
{\smallskip \global\@ignoretrue}
\makeatother

\makeatletter
{\smallskip \refstepcounter{equation}{\sc \theequation}{\sc (#1).}}%
{\smallskip \global\@ignoretrue}
\makeatother

\makeatletter
{\smallskip \refstepcounter{equation}\noindent{\sc \theequation.}{\sl{ #1.}}}%
{\smallskip \global\@ignoretrue}
\makeatother

\makeatletter
\newenvironment{borel*}%
{\smallskip \refstepcounter{equation}\noindent{\textbf{\theequation.}}}%
{\global\@ignoretrue}
\makeatother

\newcommand{\flist}[1]{\hangindent\leftmargini\textup{(1)}\hskip\labelsep {#1}%
\begin{enumerate}%
\setcounter{enumi}{1}%
}

%% file: qformsl-def.tex
\newcommand{\ot}{\otimes}

\newcommand{\darkrad}{0.115}

\newcommand{\eand}{\quad\text{and}\quad}

\DeclareMathOperator{\Lie}{Lie}

\newcommand{\C}{{\mathbb{C}}}        % the complexes
\newcommand{\R}{{\mathbb{R}}}        % the reals
\newcommand{\F}{{\mathbb{F}}}        % finite fields
\newcommand{\Z}{{\mathbb{Z}}}        % the integers
\renewcommand{\H}{{\mathcal{H}}}  % the hyperbolic plane

\newcommand{\Zm}[1]{\Z/{#1}\Z}

\newcommand{\e}{\varepsilon}

\newcommand{\D}{\Delta}

\newcommand{\G}{{\Gamma}}       % a Galois group

\newcommand{\oddots}{{\mathinner{\mkern1mu\raise1pt\vbox{\kern7pt\hbox{.}}\mkern2mu\raise4pt\hbox{.}\mkern2mu\raise7pt\hbox{.}\mkern1mu}}}

\newcommand{\qform}[1]{{\left\langle{#1}\right\rangle}}                   % a quadratic form
\DeclareMathOperator{\Spin}{Spin}           % the Spin group
\newcommand{\Sp}{\mathrm{Sp}}
\DeclareMathOperator{\SL}{SL}
\DeclareMathOperator{\GL}{GL}

\newcommand{\SO}{\mathrm{SO}}
\newcommand{\SU}{\mathrm{SU}}

\newcommand{\so}{\mathfrak{so}}

\DeclareMathOperator{\rank}{rank}
\DeclareMathOperator{\End}{End}

\newcommand{\injects}{\hookrightarrow}

\newcommand{\ra}{\rightarrow}

\newcommand{\stbtmat}[4]{\left( \begin{smallmatrix} #1&#2 \\ #3&#4 \end{smallmatrix} \right) }